\documentclass[a4paper,12pt]{article}

\usepackage{amsmath, wrapfig}
\usepackage{dsfont, a4wide, amsthm, amssymb, amsfonts, graphicx}
\usepackage{fancyhdr, xspace, psfrag, setspace, supertabular, color, enumerate}
\usepackage{hyperref}
\usepackage{bbm}

\newtheorem{theorem}{Theorem}[section]
\newtheorem{proposition}[theorem]{Proposition}
\newtheorem{lemma}[theorem]{Lemma}

\newtheorem{conjecture}[theorem]{Conjecture}

\newtheorem*{example*}{Example}

\def\d{\delta}

\def\ca{\mathcal{A}}

\def\cp{\mathcal{P}}

\def\cg{\mathcal{G}}

\def\ch{\mathcal{H}}
\def\ci{\mathcal{I}}
\def\ck{\mathcal{K}}

\def\F{\mathbb{F}}

\def\R{\mathbb{R}}

\def\N{\mathbb{N}}

\DeclareMathOperator{\Sym}{Sym}

\title{Unions of intervals in codes based on powers of sets}
\author{Thomas Karam\footnote{Mathematical Institute, University of Oxford. Email: \texttt{thomas.karam@maths.ox.ac.uk}.}}

\begin{document}
\maketitle


\begin{abstract}

We prove that for every integer $d \ge 2$ there exists a dense collection of subsets of $[n]^d$ such that no two of them have a symmetric difference that may be written as the $d$th power of a union of at most $\lfloor d/2 \rfloor$ intervals. This provides a limitation on reasonable tightenings of a question of Alon from 2023 and of a conjecture of Gowers from 2009, and investigates a direction analogous to that of recent works of Conlon, Kam\v{c}ev, Leader, R\"aty and Spiegel on intervals in the Hales-Jewett theorem.

\end{abstract}

\tableofcontents

\section{Introduction}

We will throughout use the following notations. If $n$ is a positive integer, then $[n]$ will denote the set $\{1, \dots, n\}$ of positive integers between $1$ and $n$. If $a$,$b$ are positive integers with $a \le b$, then $[a,b]$ will denote the set $\{a,a+1, \dots,b\}$ of $b-a+1$ consecutive integers starting at $a$ and ending at $b$, and we will refer to such a set as an \emph{interval}. If $A \subset B$ are two finite sets and $B$ is non-empty, then we will say that the ratio $|A|/|B|$ is the \emph{density} of $A$ inside $B$.

One research theme involves extending to patterns on set systems (and more generally to patterns on high-dimensional families) results on patterns in the integers. For instance, the Hales-Jewett theorem (proved in \cite{Hales-Jewett}) and the density Hales-Jewett theorem (proved by Furstenberg and Katznelson \cite{Furstenberg and Katznelson k=3}, \cite{Furstenberg and Katznelson}, then by the Polymath1 project \cite{Polymath}) are such generalisations of van der Waerden’s theorem and of Szemer\'edi’s theorem respectively. Likewise, the Bergelson-Leibman theorem \cite{Bergelson Leibman} is a corresponding generalisation of the polynomial van der Waerden theorem. One goal in this programme, discussed by Gowers \cite{Gowers}, is to combine these generalisations.

\begin{conjecture} \label{polynomial density Hales-Jewett conjecture} \cite[Conjecture 3]{Gowers} Let $k, d$ be positive integers and let $\d>0$. If $n$ is large enough depending on $k,d,\d$ only then for every subset $\ca$ of the set \[\ck = [k]^{[n]}\times\dots\times[k]^{[n]^{d}}\] with density at least $\d$ there is a non-empty subset $S\subset [n]$ and an element \[y \in [k]^{[n] \setminus S} \cup \dots \cup [k]^{[n]^{d} \setminus S^{d}}\] such that whenever the coordinates of $x \in \ck$ are constant on each of the sets $S, \dots, S^d$, and coincide with those of $y$ outside these sets, we have that $x \in \ca$.
\end{conjecture}

The case $k=2$ is an important step towards Conjecture \ref{polynomial density Hales-Jewett conjecture} because it is still open while at the same time avoiding the extra difficulties coming from arithmetic progressions that arise for $k \ge 3$. In turn, for $k=2$, the case where we have $\ca = \ca’ \times [2]^{[n]^{2}} \times\dots\times[2]^{[n]^{d}}$ with $\ca’ \subset [2]^n$ can immediately be seen to reduce to Sperner’s theorem \cite{Sperner}, but the “opposite” specialisation $\ca = [2]^{[n]} \times\dots\times[2]^{[n]^{d-1}} \times \ca’$ with $\ca’ \subset [2]^{[n]^{d}}$ remains a conjecture which can be reformulated as follows.

\begin{conjecture}\label{power difference conjecture}

Let $d$ be a positive integer and let $\d>0$. Then, for $n$ large enough depending on $d,\d$ only, every subset $\ca$ of $\cp([n]^d)$ that has density at least $\d$ contains a pair $(A,B)$ of distinct subsets of $[n]^d$ such that $A$ is contained in $B$ and $B \setminus A = S^d$ for some $S \subset [n]$.

\end{conjecture}

Meanwhile, a second direction of research has been to ask for strengthening of results on patterns in set systems by requiring intervals in these patterns. In the case of the Hales-Jewett theorem, substantial progress has been made, particularly for the alphabet $[3]$, which is now very well understood. Say that a \emph{combinatorial line} in $[3]^n$ is a subset $L \subset [3]^n$ such that for some strict subset $S$ of $[n]$ and some element $y \in [3]^{[n] \setminus S}$, an element $x \in [3]^n$ belongs to $L$ if and only if all coordinates of $x$ that are in $[n] \setminus S$ coincide with those of $y$ and all coordinates of $x$ that are in $S$ are the same. For every positive integer $r$, let $s(r)$ be the smallest integer such every colouring of $[3]^n$ with $r$ colours contains a monochromatic combinatorial line for which the wildcard set - that is, the set $S$ - may be written as a union of at most $s$ intervals. In a sequence of works by Conlon and Kam\v{c}ev \cite{Conlon and Kamcev}, Leader and R\"aty \cite{Leader and Raty}, and Kam\v{c}ev and Spiegel \cite{Kamcev and Spiegel}, the value of $s(r)$ has been determined for every $r$: we have $s(r) = r$ if $r$ is odd and $s(r) = r-1$ if $r$ is even.


Finally, a third avenue, described by Alon \cite{Alon} last year and which since then has received much interest and led to much follow-up work, consists in exploring questions of the following type. Given a set $\ch$ of graphs with vertex set $[n]$, we may ask for the largest size of a “graph-code” with respect to $\ch$, that is, for the largest size of a set $\ca$ of graphs with vertex set $[n]$ such that no two distinct graphs in $\ca$ have edge sets of which the \emph{symmetric difference} is a graph in $\ch$. As discussed in \cite{Alon}, one of the motivations for considering this class of problems is the case where $\ch$ is the set of all cliques with vertex set contained in $[n]$, which provides a symmetric difference version of a graph variant of the case $d=2$ of Conjecture \ref{power difference conjecture}: that graph variant had previously been suggested (\cite{Gowers}, Conjecture 4) as a possible Polymath project.

The weaker, symmetric difference version of Conjecture \ref{power difference conjecture} (that is, the statement that we obtain if we replace the two conditions $A \subset B$ and $B \setminus A = S^d$ by the single condition $A \Delta B = S^d$) is to our knowledge still open. Our main result will assert that even in this version (and therefore in Conjecture \ref{power difference conjecture} as well), the set $S^d$ cannot be required to be a union of too few intervals.

\begin{theorem}\label{No powers of unions of intervals as symmetric differences}

For all positive integers $n \ge d \ge 2$ there exists $\ca \subset \cp([n]^d)$ with $|\ca| = |\cp([n]^d)|/2$ such that there exists no pair $(A,B) \in \ca \times \ca$ satisfying $A \Delta B = S^d$ for some non-empty interval $S \subset [n]$ that can be written as a union of at most $\lfloor d/2 \rfloor$ intervals.

\end{theorem}

In Section \ref{Section: The case of one interval} we shall begin with the case of one interval, which roughly corresponds to the case $d=2$ of Theorem \ref{No powers of unions of intervals as symmetric differences}, and then in Section \ref{Section: Extending the proof to several intervals} we shall provide a full proof of Theorem \ref{No powers of unions of intervals as symmetric differences}.

As a last comment on the existing literature, we recall that the related theme of pairs of subsets of $[n]$ (rather than $[n]^d$) of which some combination avoids a specified class of sets has been studied in a systematic way in several works: this is for instance respectively the case in \cite{Leader and Long} by Leader and Long for forbidden set differences, in \cite{Karpas and Long} by Karpas and Long for forbidden symmetric differences, and in \cite{Keevash and Long} by Keevash and Long for forbidden intersections.

\section*{Acknowledgement}

The author thanks Timothy Gowers for having introduced him to Conjecture \ref{polynomial density Hales-Jewett conjecture}.

\section{The case of one interval} \label{Section: The case of one interval}

In the present section we prove the following simpler version of Theorem \ref{No powers of unions of intervals as symmetric differences}, where we only require that symmetric differences avoid powers of intervals.

\begin{theorem}\label{No square interval symmetric differences}

For all positive integers $n,d \ge 2$ there exists $\ca \subset \cp([n]^d)$ with $|\ca| = |\cp([n]^d)|/2$ such that there exists no pair $(A,B) \in \ca \times \ca$ satisfying $A \Delta B = I^d$ for some non-empty interval $I \subset [n]$.

\end{theorem}

We have chosen to do so of course for the reader’s convenience but also because later on the full proof of Theorem \ref{No powers of unions of intervals as symmetric differences} will have the same general structure, and we will then focus on the additional complications relatively to the case of Theorem \ref{No square interval symmetric differences}.

With respect to the graph-codes setting of Alon, essentially the same proof also provides a similar statement for clique symmetric differences.

\begin{theorem}\label{graph case}

Let $n \ge 2$ be an integer. There exists a subset $\cg$ of the set of (nonoriented, loopless) graphs on the vertex set $[n]$ that has density at least $1/2$ in that set and such that no two distinct graphs $G_1, G_2 \in \cg$ have edge sets the symmetric difference of which is a clique with vertex set indexed by an interval.

\end{theorem}

We now perform the successive steps that will lead us to Theorem \ref{No square interval symmetric differences}.

\begin{proof}[Proof of Theorem \ref{No square interval symmetric differences}]

We first note that it suffices to obtain $\ca$ satisfying the inequality $|\ca| \ge |\cp([n]^d)|/2$, since if this inequality is strict, then by averaging, $\ca$ must contain a pair $A$,$B$ of elements that only differ by some diagonal element, that is, some element $(a,\dots,a) \in [n]^d$ with $a \in [n]$.

Next, we note that we may reduce to the case $d=2$. To do this, we define a map $i_d: [n]^2 \to [n]^d$ by \[i_d(x,y) = i_d(x,\dots,x,y)\] for all $x,y \in [n]$. If $\ca \subset \cp([n]^2)$ then we define $\ca’ \subset \cp([n]^d)$ to be the collection of subsets of $[n]^d$ that have a restriction to \[i_d([n]^2) = \{(x,\dots,x,y): (x,y) \in [n]^2\}\] which coincides with $i_d(A)$ for some $A \in \ca$. The densities $|\ca’|/|\cp([n]^d)|$ and $|\ca|/|\cp([n]^2)|$ are the same, and if $A’,B’ \in \ca’$ satisfy $A’ \Delta B’ = I^d$ for some non-empty interval $I$ of $[n]$, then the sets $A = i_d^{-1}(A’ \cap i_d([n]^2))$ and $B = i_d^{-1}(B’ \cap i_d([n]^2))$ satisfy $A \Delta B = I^2$.

After that, we define a (nonoriented, loopless) graph $G$ with vertex set $\cp([n]^2)$ and where two distinct vertices $A$,$B$ are joined by an edge if and only if there exists an interval $I \subset [n]$ such that $A \Delta B = I^2$. In the remainder of the proof we will show that the graph $G$ is bipartite; to conclude the proof we may then take $\ca$ to be one larger of the two parts of the bipartition. (In fact, by the first observation from the present proof, they cannot have different sizes.)

In turn, to show that $G$ is bipartite it suffices to show that $G$ contains no cycle of odd length, and to establish this it suffices to prove that the family of functions \[\{\mathbbm{1}_{[a,b]^2}: 1 \le a \le b \le n\}\] is linearly independent over $\F_2$. Indeed, suppose that $G$ contains a cycle $A_0, \dots, A_k$ with $A_k=A_0$ and such that for every $i \in [k]$ we have $\{A_{i-1}, A_{i}\} \in E(G)$, that is, $A_{i} \Delta A_{i-1} = I_i^2$ for some interval $I_i \subset [n]$. Then we have\[I_1^2 \Delta \dots \Delta I_k^2 = \emptyset,\] that is, the linear relation \[\mathbbm{1}_{I_1^2} + \dots + \mathbbm{1}_{I_k^2} = 0\] over $\F_2$. The intervals $I_1, \dots, I_k$ might not yet be pairwise distinct; we consider a maximal subfamily of pairwise distinct intervals among $I_1, \dots, I_k$, which without loss of generality we can assume to be $I_1, \dots, I_l$ for some $l \le k$; letting (for each $h \in [l]$) $A_h \in \N$ be the number of appearances of $I_h$ among $I_1, \dots, I_k$ we obtain \[A_1 \mathbbm{1}_{I_1^2} + \dots + A_l \mathbbm{1}_{I_l^2} = 0\] over $\F_2$. Assuming that $\mathbbm{1}_{I_1^2}, \dots, \mathbbm{1}_{I_l^2}$ are linearly independent over $\F_2$, each of the integers $A_1, \dots, A_l$ must be even and their sum $k$ must hence be even, so the cycle $A_0, \dots, A_k$ must have even length.

Finally we establish the desired linear independence. We begin by noting that whenever $1 \le a \le b \le n$ we may write the set differences \[ [a,b]^2 \setminus [a,b-1]^2 \text{ and } [a+1,b]^2 \setminus [a+1,b-1]^2\] respectively as \begin{align*} \{(a,b), (a+1,b),\dots,(b-1,b), & (b,b), (b,b-1),\dots,(b,a+1), (b,a)\} \\ \{(a+1,b), \dots,(b-1,b), & (b,b), (b,b-1),\dots, (b,a+1)\}.\end{align*} This provides the identity \[ \mathbbm{1}_{[a,b]^2} - \mathbbm{1}_{[a,b-1]^2} - \mathbbm{1}_{[a+1,b]^2} + \mathbbm{1}_{[a+1,b-1]^2} = \mathbbm{1}_{\{(a,b), (b,a)\}},\] which holds over $\R$ and hence also over $\F_2$. Let $\Sym([n]^2)$ be the linear space of functions $u:[n]^2 \to \F_2$ satisfying $u(x,y) = u(y,x)$ for all $(x,y) \in [n]^2$. The family \[ \{\mathbbm{1}_{\{(a,a)\}}: a \in [n] \} \cup \{\mathbbm{1}_{\{(a,b), (b,a)\}}: 1 \le a < b \le n \} \] constitutes a basis of $\Sym([n]^2)$ and has size $n(n+1)/2$, and as we have now shown that it is spanned by the family \[\{\mathbbm{1}_{[a,b]^2}: 1 \le a \le b \le n\}\] which has exactly the same size, that family is also a basis of $\Sym([n]^2)$, and is in particular linearly independent over $\F_2$. \end{proof}

To prove Theorem \ref{graph case} it suffices to go through the proof of Theorem \ref{No square interval symmetric differences} while making three modifications: ignoring the first two paragraphs, ignoring the $n$ diagonal elements of $[n]^2$, and ignoring the $n$ intervals in $[n]$ with size $1$.

It is worthwhile to note how the proof of Theorem \ref{No square interval symmetric differences} fails in the case $d=1$, where \[\mathbbm{1}_{I_1 \cup I_2} - \mathbbm{1}_{I_1} - \mathbbm{1}_{I_2} = 0\] for any disjoint intervals $I_1, I_2 \subset [n]$, and how it fails to provide an analogous negative result to the case $d=2$ of Conjecture \ref{power difference conjecture}, where for any pairwise disjoint intervals $I_1, I_2, I_3 \subset [n]$ we have \[\mathbbm{1}_{(I_1 \cup I_2 \cup I_3)^2} - \mathbbm{1}_{(I_1 \cup I_2)^2} - \mathbbm{1}_{(I_1 \cup I_3)^2} - \mathbbm{1}_{(I_2 \cup I_3)^2} + \mathbbm{1}_{I_1^2} + \mathbbm{1}_{I_2^2} + \mathbbm{1}_{I_3^2} = 0.\]

This identity provides a cycle with odd length not only in the graph $G$ used in the proof of Theorem \ref{No square interval symmetric differences}, but also a cycle with odd length in the analogous graph for the case $d=2$ of Conjecture \ref{power difference conjecture}, that is, the graph with vertex set $\{0,1\}^{[n]^2}$ where we join two distinct sets $A,B$ by an edge if $A \subset B$ and $B \setminus A = S^2$ for some $S \subset [n]$. Indeed if for instance, $I_1, I_2, I_3$ are taken to be $\{1\}, \{2\}, \{3\}$ respectively, then the sequence of sets with respective indicator functions represented by the following matrices provides a cycle with length $7$. \[\begin{bmatrix} 0 & 0 & 0 \\
0 & 0 & 0 \\
0 & 0 & 0
\end{bmatrix} \begin{bmatrix} 1 & 1 & 0 \\
1 & 1 & 0 \\
0 & 0 & 0
\end{bmatrix} \begin{bmatrix} 1 & 1 & 0 \\
1 & 0 & 0 \\
0 & 0 & 0
\end{bmatrix} \begin{bmatrix} 1 & 1 & 0 \\
1 & 1 & 1 \\
0 & 1 & 1
\end{bmatrix} \begin{bmatrix} 0 & 1 & 0 \\
1 & 1 & 1 \\
0 & 1 & 1
\end{bmatrix} \begin{bmatrix} 0 & 1 & 0 \\
1 & 1 & 1 \\
0 & 1 & 0
\end{bmatrix} \begin{bmatrix} 1 & 1 & 1 \\
1 & 1 & 1 \\
1 & 1 & 1
\end{bmatrix}\]

\section{Extending the proof to unions of several intervals} \label{Section: Extending the proof to several intervals}

In the present section we prove Theorem \ref{No powers of unions of intervals as symmetric differences} in full. Several facts from the proof of Theorem \ref{No square interval symmetric differences} in Section \ref{Section: The case of one interval} extend to unions of $\lfloor d/2 \rfloor$ intervals. 

First, we had shown at the beginning of the proof of Theorem \ref{No square interval symmetric differences} that a subset $\ca \subset \cp([n]^d)$ satisfying $|\ca| > |\cp([n]^d)|/2$ must necessarily contain a pair $(A,B)$ of elements such that $A \Delta B = \{a\}^d$ for some $a \in [n]$. The singleton $\{a\}^d$ is in particular a $d$th power of a union of at most $\lfloor d/2 \rfloor$ intervals, so to prove Theorem \ref{No powers of unions of intervals as symmetric differences} we may relax the requirement $|\ca| = |\cp([n]^d)|/2$ to $|\ca| \ge |\cp([n]^d)|/2$ in its statement.

Second, we may reduce to the case where the integer $d$ is even, as follows. Suppose that $2 \le d \le D$ are two integers. We may then define a map $i_{d,D}: [n]^d \to [n]^D$ by \[i_{d,D}(x_1, \dots, x_d) = (x_1, \dots, x_1, x_2, \dots, x_d)\] for all $x \in [n]^d$, where the first coordinate is repeated $D-d+1$ times. For any given subset $\ca \subset \cp([n]^d)$ we may define a subset $\ca’ \subset \cp([n]^D)$ to be the collection of subsets of $[n]^D$ of which the restriction to \[i_{d,D}([n]^d) = \{(x_1, \dots, x_1, x_2, \dots, x_d): (x_1, \dots, x_d) \in [n]^d\}\] is equal to $i_{d,D}(A)$ for some $A \in \ca$. The density of $\ca’$ inside $\cp([n]^D)$ is the same as the density of $\ca$ inside $\cp([n]^d)$, and if $A’,B’ \in \ca’$ satisfy $A’ \Delta B’ = S^D$ for some $S \subset [n]$, then the sets \[A = i_{d,D}^{-1}(A’ \cap i_{d,D}([n]^d)) \text{ and }B = i_{d,D}^{-1}(B’ \cap i_{d,D}([n]^d))\] satisfy $A \Delta B = S^d$.

Third, like in Section \ref{Section: The case of one interval} we may define a graph $G$ with vertex set $\cp([n]^d)$, and join two distinct $A,B \in \cp([n]^d)$ by an edge if and only if $A \Delta B = S^d$ for some $S \subset [n]$ that is a union of at most $d/2$ intervals. Again like in Section \ref{Section: The case of one interval} it suffices to show that this new graph $G$ is bipartite, as we can then select $\ca$ to be either of the parts of the bipartition. Finally, writing $\ci_{k}$ for the set of non-empty subsets of $[n]$ that are unions of at most $k$ intervals \footnote{As with some other notations in this paper, the set $\ci_k$ should always be understood as depending on the integer $n$, even if this dependence will not be made explicit each time.}, to show that $G$ is bipartite it suffices to establish the following statement, which will be the main goal of the remainder of this section.

\begin{proposition}\label{linear independence in the general case}

Let $k,d,n$ be positive integers. If $n \ge d = 2k$ then the set of functions $\{\mathbbm{1}_{S^d}: S \in \ci_{k}\}$ is linearly independent over $\F_2$.

\end{proposition}

Let $d$ be a positive integer. We begin by noting that if $S$ is a subset of $[n]$, then whether an element $x \in [n]^d$ belongs to $S^d$ depends purely on the set of coordinates of $x$. Like for $d=2$ this in particular does not depend on the order of the coordinates of $x$, but for $d > 2$ it is worthwhile to note that in addition this does not depend on the multiplicities of the coordinates of $x$ either. As that holds for any $S \subset [n]$, that in particular holds for any $S \in \ci_k$.

For every set $T \subset [n]$ with $1 \le |T| \le d$ we write $V(T)$ for the set of elements $x \in [n]^d$ such that the set of values taken by the coordinates $x_1, \dots, x_d$ is equal to $T$. A reformulation of the above observation is the inclusion \[ \langle \mathbbm{1}_{S^d}: S \in \ci_k \rangle \subset \langle \mathbbm{1}_{V(T)}: 1 \le |T| \le d \rangle \] between linear subspaces (over any field, so in particular over $\F_2$). Since the functions $\mathbbm{1}_{V(T)}$ with $1 \le |T| \le d$ depend on pairwise disjoint sets of coordinates, they are linearly independent, so to establish Proposition \ref{linear independence in the general case}, it suffices to show that both families in terms of which the subspaces are defined have the same size, and that the inclusion is an equality (over $\F_2$, although our proof will not be field-dependent). We begin with the first task.

\begin{lemma}\label{equality between cardinalities}

Let $k,d,n$ be positive integers. If $2k=d$ then \[|\mathbbm{1}_{S^d}: S \in \ci_{k}| = |\{\mathbbm{1}_{V(T)}: 1 \le |T| \le d\}|. \] 

\end{lemma}

\begin{proof}

The number of sets of coordinates that an element of $[n]^d$ may have is equal to \[\binom{n}{1} + \binom{n}{2} + \dots + \binom{n}{d}. \] Meanwhile, the number of subsets of $[n]$ that are unions of exactly $k$ intervals (and no less than $k$ intervals) can be seen to be equal to $\binom{n+1}{2k}$, since we can biject the collection of these sets to the collection of subsets of $\{0,1,\dots,n\}$ with size $2k$, by sending the union of exactly $k$ intervals \[ [a_1, b_1] \cup \dots \cup [a_k,b_k]\] to the set of exactly $2k$ elements \[\{a_1-1,b_1,a_2-1,b_2, \dots, a_k-1,b_k\} \subset \{0,1,\dots,n\}.\] The number of non-empty subsets of $[n]$ that are unions of at most $k$ intervals is hence \[\binom{n+1}{2} + \binom{n+1}{4} + \dots + \binom{n+1}{2k},\] which by Pascal’s formula is equal to \[\binom{n}{1} + \binom{n}{2} + \dots + \binom{n}{2k}.\] Therefore, if $2k=d$, then the number of possible sets of coordinates of an element of $[n]^d$ is equal to the number of non-empty unions of at most $k$ intervals inside $[n]$. \end{proof}

Given Lemma \ref{equality between cardinalities}, proving Proposition \ref{linear independence in the general case} reduces to proving the following statement.

\begin{proposition}\label{spanning in the general case}

Let $k,d,n$ be positive integers. If $n \ge d = 2k$ then for every $T \subset [n]$ with $1 \le |T| \le d$ the function $\mathbbm{1}_{V(T)}$ is spanned by the functions $\mathbbm{1}_{S^d}$ with $S \in \ci_{k}$.

\end{proposition}

To establish Proposition \ref{spanning in the general case} we will resort to an auxiliary lemma, in which we will use the following notations. If $\omega$ is an element of $\{0,1\}^d$ for some integer $d$, then we write $|\omega|$ for the number of coordinates of $\omega$ that are equal to $1$. If $E$, $J$ are subsets of $[n]$ then we write $V(E,J)$ for the set of $x \in [n]^d$ such that the elements of $E$ all appear in the coordinates of $x$, and all coordinates of $x$ belong to $J$. More succinctly, an element $x \in [n]^d$ belongs to $V(E,J)$ if it belongs to $V(T)$ for some set $T$ satisfying $E \subset T \subset J$. If $L$ is a subset of $[d]$ then we write $\{0,1\}^L$ for the set of $\omega \in \{0,1\}^d$ such that $\omega_i = 0$ for every $i \in [d] \setminus L$.

\begin{lemma}\label{identity with omega}

Let $n \ge d$ be positive integers with $d$ even, let $a_1 < \dots < a_d$ be elements of [n], let $J$ be the union of intervals \[[a_1, a_2] \cup \dots \cup [a_{d-1}, a_d],\] and let $E = \{a_l: l \in L\}$ for some $L \subset [d]$. Then we have the decomposition \begin{equation} \sum_{\omega \in \{0,1\}^{L}} (-1)^{|\omega|} \mathbbm{1}_{([a_1 + \omega_1, a_2 - \omega_2] \cup \dots \cup [a_{d-1} + \omega_{d-1}, a_d - \omega_d])^d} = \mathbbm{1}_{V(E,J)}. \label{generating indicators of value sets} \end{equation}

\end{lemma}

\begin{proof}

We will assume without loss of generality that $L = [l]$ for some $0 \le l \le d$: as will be clear from the remainder of the proof all that we will be using about $a_1, \dots, a_l$ is that they are endpoints of the union $J$ of intervals, at least in the way that we defined it as a union of intervals (they might not be genuine endpoints, if $a_{2t+1}=a_{2t}+1$ for some $1 \le t \le d/2 - 1$, but this will not be an impediment for us). We begin by writing \[ \mathbbm{1}_{([a_1, a_2] \cup \dots \cup [a_{d-1}, a_d])^d} - \mathbbm{1}_{([a_1+1, a_2] \cup \dots \cup [a_{d-1}, a_d])^d} = \mathbbm{1}_{V(\{a_1\}, J)}.\] Replacing $a_2$ by $a_2-1$ we likewise have that \[\mathbbm{1}_{([a_1, a_2-1] \cup \dots \cup [a_{d-1}, a_d])^d} - \mathbbm{1}_{([a_1+1, a_2-1] \cup \dots \cup [a_{d-1}, a_d])^d} = \mathbbm{1}_{V(\{a_1\}, J \setminus \{a_2\})}.\] Subtracting the second identity from the first provides \[\sum_{\omega \in \{0,1\}^2} (-1)^{|\omega|}\mathbbm{1}_{([a_1 + \omega_1, a_2 - \omega_2] \cup [a_3, a_4] \cup \dots \cup [a_{d-1},a_d])^d} = \mathbbm{1}_{V(\{a_1, a_2\}, J)}.\] Replacing $a_3$ by $a_3+1$ in this equality we obtain the identity \[\sum_{\omega \in \{0,1\}^2} (-1)^{|\omega|}\mathbbm{1}_{([a_1 + \omega_1, a_2 - \omega_2] \cup [a_3+1, a_4] \cup \dots \cup [a_{d-1},a_d])^d} = \mathbbm{1}_{V(\{a_1, a_2\}, J \setminus \{a_3\})},\] and subtracting it from the identity before it then yields \[\sum_{\omega \in \{0,1\}^3} (-1)^{|\omega|}\mathbbm{1}_{([a_1 + \omega_1, a_2 - \omega_2] \cup [a_3 + \omega_3, a_4] \cup [a_5, a_6] \cup \dots \cup [a_{d-1},a_d])^d} = \mathbbm{1}_{V(\{a_1, a_2, a_3\}, J)}.\] Continuing like this we eventually conclude that $\mathbbm{1}_{V(\{a_1, a_2, a_3, \dots, a_l\}, J)}$ is equal to \[\sum_{\omega \in \{0,1\}^{l}} (-1)^{|\omega|} \mathbbm{1}_{([a_1 + \omega_1, a_2 - \omega_2] \cup \dots \cup [a_{l-1} + \omega_{l-1}, a_l - \omega_l] \cup [a_{l+1}, a_{l+2}] \cup \dots \cup [a_{d-1}, a_d])^d} \] if $l$ is even and to \[\sum_{\omega \in \{0,1\}^{l}} (-1)^{|\omega|} \mathbbm{1}_{([a_1 + \omega_1, a_2 - \omega_2] \cup \dots \cup [a_{l-2} + \omega_{l-2}, a_{l-1} - \omega_{l-1}] \cup [a_l + \omega_l, a_{l+1}] \cup [a_{l+2}, a_{l+3}] \cup \dots \cup [a_{d-1}, a_d])^d}\] if $l$ is odd; in both cases this is the desired identity. \end{proof}

Repeatedly applying this identity to sets $L$ of decreasing size now allows us to finish the deduction of Proposition \ref{spanning in the general case}. As stated earlier, this finishes the proof of Proposition \ref{linear independence in the general case} and therefore of Theorem \ref{No powers of unions of intervals as symmetric differences}.

\begin{proof}[Proof of Proposition \ref{spanning in the general case}]

We establish by decreasing induction on $|T|$ that if $1 \le |T| \le d$ then $\{\mathbbm{1}_{S^d}: S \in \ci_{k}\}$ spans $\mathbbm{1}_{V(T)}$. The base case is the case $|T|=d$. There, writing $T = \{a_1, \dots, a_d\}$ with $a_1 < \dots < a_d$ and applying Lemma \ref{identity with omega} in the case $L = [d]$ we obtain \[\sum_{\omega \in \{0,1\}^{d}} (-1)^{|\omega|} \mathbbm{1}_{([a_1 + \omega_1, a_2 - \omega_2] \cup \dots \cup [a_{d-1} + \omega_{d-1}, a_d - \omega_d])^d} = \mathbbm{1}_{V(\{a_1, a_2, a_3, \dots, a_d\}, J)},\] and as an element of $[n]^d$ may have only at most $d$ pairwise distinct coordinates, the right-hand side is exactly $\mathbbm{1}_{V(T)}$ as desired.

Suppose next that $|T|=d-1$, which we write as $T = \{a_1, \dots, a_{d-1} \}$ with $a_1 < \dots < a_{d-1}$. Since $n \ge d$, we may complete $T$ into a set $T’ = \{a_1’, \dots, a_d’\}$ of size $d$ with $a_1’< \dots < a_d’$ by adding an arbitrary extra value between some two of the values $a_1,\dots, a_{d-1}$. Let $i \in [d]$ be the index of this extra value. Applying Lemma \ref{identity with omega} with $L = [d] \setminus \{i\}$ provides the identity \[\sum_{\omega \in \{0,1\}^{d}: \omega_i = 0} (-1)^{|\omega|} \mathbbm{1}_{([a_1’ + \omega_1, a_2’ - \omega_2] \cup \dots \cup [a_{d-1}’ + \omega_{d-1}, a_d’ - \omega_d])^d} = \mathbbm{1}_{V(\{a_1', \dots, a_{i-1}’, a_{i+1}’ \dots, a_d’\}, J’)},\] where $J’ = [a_1’, a_2’] \cup \dots \cup [a_{d-1}’, a_d’]$. The right-hand side may in turn be decomposed: we have \[\mathbbm{1}_{V(\{a_1’, \dots, a_{i-1}’, a_{i+1}’ \dots, a_d’\}, J’)} = \mathbbm{1}_{V(T)} + \sum_{u \in J' \setminus \{a_1’,\dots,a_d’\}} \mathbbm{1}_{V(\{a_1’, \dots, a_{i-1}’, a_{i+1}’ \dots, a_d’,u\})}.\] The last two decompositions together provide a decomposition of $\mathbbm{1}_{V(T)}$ in terms of sets of the type $\mathbbm{1}_{V(T'')}$ with $|T''| = d$, which by the base case of the induction ensures that $\mathbbm{1}_{V(T)}$ is spanned by $\{\mathbbm{1}_{S^d}: S \in \ci_{k}\}$.

More generally, the inductive step is as follows. Suppose that for some $1 \le l < d$ we have shown that if $T$ is a subset of $[n]$ with size $l+1 \le |T| \le d$ then $\mathbbm{1}_{V(T)}$ is spanned by $\{\mathbbm{1}_{S^d}: S \in \ci_{k}\}$, and that we want to deduce from this that this still holds for $|T|=l$. Let $T$ be a subset of $[n]$ with size $l$. We may write $T = \{a_1, \dots, a_l\}$ with $a_1 < \dots < a_l$. The subset $T$ may be extended into a set $T’ = \{a_1’, \dots, a_d’\}$ of elements of $[n]$, with $a_1' < \dots < a_d'$. Writing $L$ for the set of indices $i \in [d]$ such that $a_i$ belongs to $T$, Lemma \ref{identity with omega} provides \begin{equation} \sum_{\omega \in \{0,1\}^{L}} (-1)^{|\omega|} \mathbbm{1}_{([a_1 + \omega_1, a_2 - \omega_2] \cup \dots \cup [a_{d-1} + \omega_{d-1}, a_d - \omega_d])^d} = \mathbbm{1}_{V(T, T’)}. \label{identity with V(T,T')} \end{equation} Recall that the condition for an element $x \in [n]^d$ to belong to $V(T,T’)$ is that its coordinates take every value in $T$, and all its coordinates take values in $T’$; this is equivalent to saying that the set of values taken by the coordinates of $x$ contains $T$ and is contained in $T’$; more formally we have the identity \[\mathbbm{1}_{V(T, T’)}  = \sum_{T \subset T'' \subset T’} \mathbbm{1}_{V(T'')}\] which we may rewrite as \begin{equation} \mathbbm{1}_{V(T)} = \mathbbm{1}_{V(T, T’)} - \sum_{T \subset T'' \subset T’: T'' \neq T} \mathbbm{1}_{V(T'')}. \label{decomposition of V(T) in the inductive step} \end{equation}

The identity \eqref{identity with V(T,T')} shows that $\{\mathbbm{1}_{S^d}: S \in \ci_{k}\}$ spans $\mathbbm{1}_{V(T, T’)}$, and since the sets $T''$ over which the sum is taken on the right-hand side of \eqref{decomposition of V(T) in the inductive step} all have size between $l+1$ and $d$, the inductive hypothesis shows that they are each spanned by $\{\mathbbm{1}_{S^d}: S \in \ci_{k}\}$ as well. This set therefore spans $\mathbbm{1}_{V(T)}$, which completes the inductive step. \end{proof}

Proposition \ref{linear independence in the general case} no longer holds if we instead consider unions of at most $k > \lfloor d/2 \rfloor$ intervals, as then (provided that $n \ge d+1$) the equality in Lemma \ref{equality between cardinalities} becomes the strict inequality \[|\mathbbm{1}_{S^d}: S \in \ci_{k}| > |\{\mathbbm{1}_{V(T)}: 1 \le |T| \le d\}|. \] Furthermore, for unions of at most $k>\lfloor d/2 \rfloor$ intervals there is a more basic obstacle to the present proof method, as we can directly describe a cycle with odd length $2^{d+1} - 1$ in the resulting graph $G$ which is hence no longer bipartite. Indeed we have the identity \[ \sum_{\emptyset \neq S \subset [d+1]} (-1)^{|S|} \mathbbm{1}_{S^d} = 0\] as can be checked pointwise: for any given $x \in [n]^d$, letting $T$ be the set of values taken by the coordinates $x_1, \dots, x_d$, the sum evaluates at $x$ to \[\sum_{T \subset S \subset [d+1]} (-1)^{|S|} = (-1)^{|T|} \sum_{U \subset [d+1] \setminus T} (-1)^{|U|}\] which cancels since $[d+1] \setminus T$ is non-empty. Since every subset of $[d+1]$ is a union of at most $\lceil (d+1)/2 \rceil = \lfloor d/2 \rfloor+1$ intervals, this obstacle comes into play as soon as $k > \lfloor d/2 \rfloor$.

\section{Final comments}

Suppose that for some positive integers $d$,$a$ and some sequence of functions $F_n: \cp([n])^a \to \cp([n]^d)$ we consider the following statement: that if $\d>0$ and $\ca$ is a collection of at least a proportion $\d$ of the subsets of $[n]^d$, then for $n$ large enough (depending on $\d$) we can find a pair $A,B$ of distinct elements of $\ca$ satisfying \[A \Delta B = F_n(S_1, \dots, S_a)\] for some subsets $S_1, \dots, S_a \subset [n]$. (In this paper we have focused on the case $a=1$ and $F_n:S \mapsto S^d$.) There does not yet appear to be a known basic principle which allows one to characterise when, assuming that this statement is true, we may furthermore require the sets $S_1, \dots, S_a$ to be intervals or unions of a bounded number of intervals, with this bound depending on $\d$ only.

On the other hand, provided that the sequence of functions $(F_n)$ is nested in the sense that $F_{n_1}(S_1, \dots, S_a) = F_{n_2}(S_1, \dots, S_a)$ for all integers $1 \le n_1 \le n_2$ and all subsets $S_1, \dots, S_a \subset [n_1]$, there is a simple way to achieve such a bound that is allowed to depend on $\d$ but is besides that uniform in $n$ and in $\ca$. We end by describing this simple argument; we have chosen to write it specifically for the situation that we have considered throughout this paper but it may be adapted line by line to the more general setting that we have just mentioned.

\begin{proposition} \label{equivalence between three statements}

For every integer $d \ge 1$ and every $\d>0$ the following three statements are equivalent.

\begin{enumerate}[(i)]

\item There exists an integer $N_1(\d)$ satisfying if $n \ge N_1(\d)$ then every $\ca \subset \cp([n]^d)$ satisfying $|\ca| \ge \d |\cp([n]^d)|$ contains a pair $(A,B)$ such that $A \Delta B = S^d$ for some non-empty $S \subset [n]$.

\item There exist integers $N_2(\d)$, $T_2(\d)$ satisfying if $n \ge N_2(\d)$ then every $\ca \subset \cp([n]^d)$ satisfying $|\ca| \ge \d |\cp([n]^d)|$ contains a pair $(A,B)$ such that $A \Delta B = S^d$ for some non-empty $S \subset [n]$ which is the union of at most $T_2(\d)$ intervals.

\item There exist integers $N_3(\d)$, $T_3(\d)$ satisfying if $n \ge N_3(\d)$ then every $\ca \subset \cp([n]^d)$ satisfying $|\ca| \ge \d |\cp([n]^d)|$ contains a pair $(A,B)$ such that $A \Delta B = S^d$ for some non-empty $S \subset [n]$ with size at most $T_3(\d)$.

\end{enumerate} Quantitatively, (ii) implies (i) with $N_1(\d) \le N_2(\d)$, (iii) implies (ii) with $N_2(\d) \le N_3(\d)$ and $T_2(\d) \le T_3(\d)$, and (i) implies (iii) with $N_3(\d) \le N_1(\d)$ and $T_3(\d) \le N_1(\d)$.

\end{proposition}

\begin{proof}

The first two implications are immediate; let us show the third. Suppose that (i) holds for some $N_1(\d)$. Let $\ca \subset \cp([n]^d)$ such that $|\ca| \ge \d \cp([n]^d)$ for some $\d>0$ and some integer $n \ge N_1(\d)$. Then by averaging, there exists a subset $U \subset [n]^d \setminus [N_1(\d)]^d$ such that the density $|\ca’| / |[N_1(\d)]^d|$ of \[\ca’= \{A' \subset [N_1(\d)]^d: A' \cup U \in \ca\}\] inside $[N_1(\d)]^d$ is at least the density $|\ca| / |[n]^d|$ of $\ca$ inside $[n]^d$. By (i) we can find a pair of subsets $A’,B’ \subset [N_1(\d)]^d$ such that $A' \Delta B' = S^d$ for some non-empty subset $S \subset [N_1(\d)]$, and the sets $A,B \in \cp([n]^d)$ defined by $A = A’ \cup U$ and $B = B’ \cup U$ then both belong to $\ca$ and satisfy $A \Delta B = S^d$. Furthermore, the set $S$ has size at most $N_1(\d)$. \end{proof}

\end{document}